\documentclass[12pt]{amsart}

\title{On Maximum, Typical and Generic Ranks}
\date{\today}

\author{Grigoriy Blekherman}
\address{School of Mathematics \\
Georgia Tech \\
686 Cherry Street \\
Atlanta, GA 30332 \\
USA}
\email{greg@math.gatech.edu}
\author{Zach Teitler}
\address{Department of Mathematics \\
1910 University Drive \\
Boise State University \\
Boise, ID 83725-1555 \\
USA}
\email{zteitler@boisestate.edu}



\usepackage[letterpaper,margin=1in]{geometry}
\usepackage{booktabs}
\usepackage{xcolor}
\usepackage{hyperref}

\newtheorem{theorem}{Theorem}
\newtheorem{corollary}[theorem]{Corollary}
\newtheorem{proposition}[theorem]{Proposition}
\newtheorem{lemma}[theorem]{Lemma}

\theoremstyle{definition}
\newtheorem{example}[theorem]{Example}
\newtheorem{definition}[theorem]{Definition}

\theoremstyle{remark}
\newtheorem{remark}[theorem]{Remark}

\newcommand{\PP}{\mathbb{P}}

\newcommand{\C}{\mathbb{C}}
\newcommand{\R}{\mathbb{R}}
\newcommand{\F}{\mathbb{F}}

\newcommand{\rgen}{r_{\mathrm{gen}}}
\newcommand{\rmax}{r_{\mathrm{max}}}

\DeclareMathOperator{\Sym}{Sym}

\newcommand{\defining}[1]{\textbf{#1}}

\begin{document}

\begin{abstract}
We show that for several notions of rank including tensor rank, Waring rank,
and generalized rank with respect to a projective variety,
the maximum value of rank is at most twice the generic rank.
We show that over the real numbers, the maximum value of the real rank is at most twice the smallest typical rank,
which is equal to the (complex) generic rank.
\end{abstract}

\maketitle

\section{Introduction}

In many areas of applied mathematics, machine learning, and engineering,
the notion of shortest decomposition of a vector into simple vectors is of prime importance.
See for example \cite{comonmour96,MR2736103,MR2895192,MR2865915} and \cite[Chapter 4]{MR2723140}.
The length of the shortest decomposition is usually called the \defining{rank} of the vector.

In this article we consider the rank of a vector with respect to a variety
over an arbitrary field $\F$, but we will highlight
the real and complex situations later on.
Let $X\subset \F\PP^n$ be a projective variety
and let $\hat{X}\subset \F^{n+1}$
be the affine cone over $X$.
The variety $X$ is called \defining{nondegenerate} if $X$ (or equivalently $\hat{X}$) is not contained in a hyperplane.
In this case, for any vector $v \in \F^{n+1}$, $v \neq 0$,
we can define the \defining{rank of $v$ with respect to $X$} (\defining{$X$-rank} of $v$ for short)
as follows:
\[
  \operatorname{rank}_X(v) = \min r
  \quad
  \text{such that}
  \quad
  v=\sum_{i=1}^r x_i \quad \text{where} \quad x_i \in \hat{X},
\]
i.e., the $X$-rank of $v$ is the length of the shortest decomposition of $v$ into elements of $\hat{X}$.
We will assume that the variety $X$ is irreducible, as is the case in the applications of interest.

For example, tensor rank (real or complex) is rank with respect to the Segre variety,
symmetric tensor rank (also called Waring rank) is rank with respect to the Veronese variety, and anti-symmetric tensor rank
is rank with respect to the Grassmannian variety.
See section~\ref{sec:tensors} for more discussion of examples.

A rank $r$ is called \defining{generic} if the vectors of $X$-rank $r$ contain
a Zariski open subset of $\F^{n+1}$.
For this we assume $\F$ is infinite, to avoid trivialities.
It is well-known that over any algebraically closed field there is a unique generic $X$-rank for a nondegenerate variety $X$.
Over $\C$,  we can equivalently define rank $r$ to be generic if the set of vectors of rank $r$ contains an open subset
of $\C^{n+1}$ with respect to the standard product topology.
A significant effort has gone into the calculation of the generic rank for various varieties $X$,
and the generic rank is fairly well understood for various tensor ranks over $\C$, see section~\ref{sec:tensors}.

Much less is known about the maximum $X$-rank.
One of our main results shows that the generic rank and the maximum rank cannot be too far apart:

\begin{theorem}\label{thm:mainC}
Let $X \subset \F\PP^n$ be an irreducible nondegenerate variety over an algebraically closed field $\F$.
Let $\rmax$ be the maximum value of rank with respect to $X$ and
let $\rgen$ be the generic value of rank with respect to $X$.
Then $$\rmax \leq 2 \rgen.$$
\end{theorem}

While the above bound is elementary, we will show in Section \ref{sec:tensors} that it strongly improves
many existing bounds on the maximum tensor rank.
Also, due to its generality it applies to any notion of tensor rank: regular, symmetric and antisymmetric.
We show in Theorem \ref{thm: refinement if secant hypersurface} that the above bound can be slightly improved
to $\rmax \leq 2 \rgen-1$ in the special case where the Zariski closure of the vectors of rank $\rgen-1$
forms a hypersurface in $\F^{n+1}$ and $\F$ has characteristic $0$.

\medskip

Over the real numbers a rank $r$ is called \defining{typical} if the set of vectors of rank $r$ contains an open subset
of $\R^{n+1}$ with respect to the Euclidean topology.
Unlike over the complex numbers, there may exist more than one typical rank.
However, the lowest typical rank is equal to the generic rank over the complexification of $X$.
While this is probably well-known to the experts, we did not find it explicitly stated
in the literature in this generality:

\begin{theorem}\label{thm:typ=gen}
Let $X \subset \R\PP^n$ be an irreducible nondegenerate real projective variety
whose real points are Zariski dense.
Let $r_0$ be the minimum typical rank with respect to $X$
and let $\rgen$ be the generic rank with respect to the complexification $X_\C = X \otimes \C$.
Then
\[
  r_0 = \rgen.
\]
\end{theorem}

We are able to show a very similar and elementary bound for maximum $X$-rank with respect to real varieties as well:

\begin{theorem}\label{thm:mainR}
Let $X \subset \R\PP^n$ be an irreducible nondegenerate real projective variety
whose real points are Zariski dense.
Let $r_0$ be the minimum typical rank with respect to $X$, and
let $\rmax$ be the maximum value of rank with respect to $X$.
Then $$\rmax \leq 2 r_0.$$
\end{theorem}

Again, despite being elementary the above theorem provides the best known bounds on maximum real tensor rank in many cases.
In fact, much less is known about the maximum real rank, so the bound of Theorem \ref{thm:mainR} is even stronger.

\section{Proofs of Main Theorems}

We begin by working over an algebraically closed field, then we work over $\R$.

\subsection{Rank over algebraically closed fields.}

Let $X \subset \PP^n$ be an irreducible nondegenerate variety over an algebraically closed field $\F$.
Since $X$ spans $\PP^n$, choosing a basis from points of $X$ shows that every point has rank at most $n+1$.
The best general bound on maximum rank with respect to $X$ is to our knowledge, in characteristic $0$:
\begin{equation}\label{eq: general bound}
  \rmax(X) \leq n+1-\dim(X),
\end{equation}
see \cite{Geramita}, remarks following Theorem~7.6,
and \cite[Prop.~5.1]{Landsberg:2009yq}.
This fails in positive characteristic.
For example, if $X$ is a smooth plane conic in characteristic $2$ and $P$ is its strange point,
meaning that every line through $P$ is tangent to $X$,
then $r(P) = 3$.
However see \cite{MR2783180} for the result $\rmax(X) \leq n+2-\dim(X)$ in positive characteristic.

We now prove Theorem \ref{thm:mainC}, which usually provides a much better bound on the maximum rank:

\begin{proof}[Proof of Theorem \ref{thm:mainC}]
Every point is a sum of two general points, each having rank $\rgen$, and therefore the maximum rank is at most $2\rgen$.

Explicitly, let $U$ be a Zariski dense open subset of points of rank exactly $\rgen$.
Let $q \in \PP^n$ be any point and let $p$ be any point in $U$.
The line $L$ through $q$ and $p$ intersects $U$ at another point $p'$
(in fact, at infinitely many more points).
Since $p$ and $p'$ span $L$, $q$ is a linear combination of $p$ and $p'$.
Since $p$ and $p'$ each have rank $\rgen$, $q$ has rank at most $2\rgen$.
\end{proof}

We now describe a slight improvement to this theorem in certain cases.
First we recall the definition of the secant variety.

\begin{definition}
The $r$-th \defining{secant variety} $\sigma_r(X)$ is the Zariski closure of the set of points of rank at most $r$.
\end{definition}
Since points of $X$ span $\PP^n$ the $(n+1)$-st secant variety $\sigma_{n+1}(X)$ certainly fills $\PP^n$.
Note that the generic rank with respect to $X$ is the least $r$ such that $\sigma_r(X) = \PP^n$.

The following map will be useful.
Let $\hat{X}$ be the affine cone over $X$
and let $\Sigma_{r,X} : \hat{X}^r \to \mathbb{A}^{n+1}$ be the map
\[
  \Sigma_{r,X}(x_1,\dotsc,x_r) = x_1 + \dotsb + x_r .
\]
The image of $\Sigma_{r,X}$ is precisely the affine cone over the set of points of rank $r$ or less.
The secant variety $\sigma_r(X)$ is the Zariski closure of the projectivization of $\Sigma_{r,X}(\hat{X}^r)$.
(See for example \cite{MR2252121} where a slight variant of this map is used to describe the defining ideal of $\sigma_r(X)$.)
If $X$ is irreducible then so is $\sigma_r(X)$ and
if $r < \rgen(X)$ then $\dim \sigma_r(X) < \dim \sigma_{r+1}(X)$ \cite[1.2]{MR947474}.

The set of points of rank $\rgen$ contains a dense Zariski open subset of $\sigma_{\rgen}(X) = \PP^n$.
For $r < \rgen$ the set of points of rank $\leq r$ contains a dense subset of $\sigma_r(X)$;
this dense subset can be taken to be open, and to consist of points of rank equal to $r$:
\begin{lemma}
Let $X \subset \PP^n$ be an irreducible nondegenerate variety over an algebraically closed field.
Let $r \leq \rgen$.
The set of points of rank equal to $r$ contains a dense Zariski open subset of $\sigma_r(X)$.
\end{lemma}
\begin{proof}
The projectivization of the image of $\Sigma_{r,X}$ is dense in $\sigma_r(X)$ and constructible by Chevalley's theorem.
Every dense constructible set contains a dense open subset.
The set of points of rank less than $r$ is contained in $\sigma_{r-1}(X)$,
which has strictly lower dimension than $\sigma_r(X)$.
Removing it leaves a nonempty dense open subset of the points of rank equal to $r$.
\end{proof}

Now we can give a slight improvement to Theorem \ref{thm:mainC}
if the secant variety $\sigma_{\rgen-1}(X)$ is a hypersurface,
in characteristic zero.

\begin{theorem}\label{thm: refinement if secant hypersurface}
Let $X \subset \PP^n$ be an irreducible nondegenerate variety over an algebraically closed field of characteristic zero.
Suppose $X$ is not a hypersurface,
but for some $r$, $\sigma_r(X)$ is a hypersurface;
necessarily $r = \rgen - 1$.
Then $\rmax \leq 2r+1 = 2\rgen - 1$.
If the hypersurface $\sigma_r(X)$ has no points of multiplicity equal to $\deg(\sigma_r(X))-1$
then $\rmax \leq 2r = 2\rgen - 2$.
\end{theorem}

\begin{proof}
Let $d = \deg(\sigma_r(X))$.
First suppose $q$ is any point at which $\sigma_r(X)$ has multiplicity strictly less than $d-1$,
including multiplicity $0$ if $q$ lies off of $\sigma_r(X)$.
Let $p \in \sigma_r(X)$ be a general point and let $L$ be the line through $p$ and $q$.
By Bertini's theorem \cite[Thm.~6.3]{MR725671} $L$ intersects $\sigma_r(X)$ with multiplicity $1$ at $p$,
and with multiplicity strictly less than $d-1$ at $q$ (if $q$ lies on $\sigma_r(X)$).
So $L$ intersects $\sigma_r(X)$ in at least one more point $p'$ distinct from $p$ and $q$.
Since $p$ is general, so is $p'$.
Then $q$ is a linear combination of $p$ and $p'$, so $r(q) \leq r(p) + r(p') = 2r$.

Next suppose $q$ is a point at which $\sigma_r(X)$ has multiplicity $d$.
Then $\sigma_r(X)$ is a cone with vertex $q$.
Let $p \in \sigma_r(X)$ be a general point and $L$ the line through $p$ and $q$;
then $L \subset \sigma_r(X)$ and $L$ intersects the open subset of points of rank $r$,
so in fact $L$ has infinitely many points of rank $r$.
Choosing $p'$ as before, once again $r(q) \leq 2r$.

Finally suppose $q$ is a point at which $\sigma_r(X)$ has multiplicity equal to $d-1$.
Let $p \in \PP^n$ be a general point and let $L$ be the line through $p$ and $q$.
Since $p$ is general, $L$ intersects $\sigma_r(X)$ with multiplicity $d-1$ at $q$,
so $L$ intersects $\sigma_r(X)$ at another point $p'$.
Since $p$ is general, $p' \in \sigma_r(X)$ is general.
Then $q$ is a linear combination of $p$ and $p'$, so $r(q) \leq r(p) + r(p') = 2r+1$.
\end{proof}

We mention a simple generalization of Theorem~\ref{thm: refinement if secant hypersurface}
and \eqref{eq: general bound}:
\begin{proposition}
Suppose $\F$ has characteristic $0$ and $\sigma_k(X)$ has codimension $c$.
Let $s$ be the maximum rank of points on $\sigma_k(X)$.
Then $\rmax \leq \max\{s,(c+1)k\}$.
\end{proposition}
\begin{proof}
Let $q \in \PP^N$.
If $q \in \sigma_k(X)$ then $r_X(q) \leq s$.
Otherwise, a general $c$-plane through $q$ is spanned by its intersection with $\sigma_k(X)$,
which is reduced by Bertini's theorem \cite[Thm.~6.3]{MR725671},
giving $q$ as a linear combination of $c+1$ general points on $\sigma_k(X)$ which each have rank $k$.
\end{proof}

The upper bound \eqref{eq: general bound} is given by $k=s=1$.
Theorem~\ref{thm: refinement if secant hypersurface} is the case $c=1$. 
We believe that the best bounds resulting from this proposition are just these previously observed extreme cases,
and intermediate values of $k$ and $c$ probably do not give interesting new bounds.
Perhaps if some secant variety of $X$ is highly degenerate,
this bound might be interesting.

\begin{remark}
The $X$-rank of $q$ is the least length of a reduced zero-dimensional subscheme of $X$
whose span includes $q$.
A zero-dimensional subscheme of $X$ whose span includes $q$ is called an apolar scheme to $q$.
Related quantities include the cactus rank (or scheme length), the least length of any apolar scheme;
the smoothable rank, the least length of any smoothable apolar scheme;
the curvilinear rank, similarly; and so on.
See \cite{Bernardi:2012fk} for a thorough treatment and comparison of these and several other notions.
In \cite{Ballico:2013kx} it is shown that the cactus rank, smoothable rank, and curvilinear rank
with respect to a Veronese variety are bounded by twice the generic rank.
As all these quantities are less than or equal to the $X$-rank, we recover these results.
\end{remark}

\subsection{Real varieties}

Now we consider rank with respect to a real variety.
As before, let $X \subset \R\PP^n$ be an irreducible nondegenerate variety.
We assume the real points of $X$ are Zariski dense in $X$,
that is, $X$ is the Zariski closure of its set of real points;
equivalently, $X$ has a smooth real point.

Let $X_\C = X \otimes \C$ be the complexification of $X$, i.e.\ the variety in $\C\PP^n$
defined by the same equations that define $X \subset \R\PP^n$.
Since $X_\C$ is still the Zariski closure of the real points of $X$, $X_\C$ is irreducible.
(Otherwise, if $f, g$ are complex polynomials such that $fg$ vanishes on $X_\C$ but neither $f$ nor $g$ does,
then $|f|^2$ and $|g|^2$ are real polynomials inducing a decomposition of the real points of $X$.)
A priori, the real rank of $v \in \R^{n+1}$ with respect to $X$ may be strictly greater than
the (complex) rank of the same point $v \in\C^{n+1}$  with respect to the complexification $X_\C$.
See for instance Example 1.2 of \cite{Reznick:2013uq}.
Moreover, the maximum real rank with respect to $X$ may a priori be different than the maximum (complex)
rank with respect to $X_\C$ (but we are not aware of an example in which this happens).

An integer $r$ is called a \defining{typical rank} if it is the rank of every point in some nonempty open subset of $\R^{n+1}$
in the Euclidean topology.
In contrast to the closed field case, there may be more than one typical rank.
See for example \cite{Blekherman:2012dq}.

We now show Theorem \ref{thm:typ=gen}.
It is proved for triple tensor products $\C^\ell \otimes \C^m \otimes \C^n$
in \cite[Thm.~7.1]{MR2854886}.
(See also results and references in \cite[\textsection7]{MR2854886}
regarding the maximum typical rank.)

\begin{proof}[Proof of Theorem \ref{thm:typ=gen}]
Certainly $r_0 \geq \rgen$: $\sigma_{\rgen-1}(X_\C)$
is contained in a hypersurface
and it is defined over $\R$.
Therefore $\sigma_{\rgen-1}(X_\C)$ is contained in a hypersurface defined over $\R$,
and hence so is $\sigma_{\rgen-1}(X)$.
Thus $\rgen-1$ is not a typical rank, nor is any rank less than $\rgen-1$.

On the other hand, $r = \rgen$ is a typical rank.
Since $\hat{X}$ and $\hat{X}^r$ are semialgebraic sets
and the map $\Sigma_{r,X}$ is a linear projection,
the image $\mathcal{S} = \Sigma_{r,X}(\hat{X}^r)$ is a semialgebraic set by the Tarski-Seidenberg theorem \cite[Thm.~2.2.1]{MR1659509}.
We can write $\mathcal{S}$ as a finite union $B_1 \cup \dotsb \cup B_t$ of basic semialgebraic sets,
where each $B_i$ is nonempty and defined by a finite set of real polynomial equations $f(x) = 0$
and inequalities $f(x) > 0$ \cite[\textsection2.1]{MR1659509}.
If the definition of $B_i$ includes an equation then $B_i$ is contained in a hypersurface.
Since the Zariski closure of $\mathcal{S}$ is $\sigma_{\rgen}(X) = \R\PP^n$,
$\mathcal{S}$ is Zariski dense,
so there must be at least one $B_i$ whose definition
consists solely of inequalities.
Removing the closure of points of rank less than $\rgen$ if necessary, this $B_i$ is an open set in the Euclidean topology
consisting of points of rank $\rgen$.
This shows $\rgen$ is a typical rank.
\end{proof}

Combining the above result with Theorem~\ref{thm:mainC}
shows that the maximum complex rank with respect to $X$ is at most twice the lowest typical rank.
But as we remarked above, the maximum real rank may be different than the maximum complex rank.
So in Theorem~\ref{thm:mainR} we show, analogously to Theorem~\ref{thm:mainC},
that the maximum real rank with respect to $X$ is also bounded by twice the least typical rank.
\begin{proof}[Proof of Theorem \ref{thm:mainR}]
Let $B \subset \R^{n+1}$ be a small open ball in which every point has rank $r_0$.
Then $B-B$ is an open neighborhood of the origin in which every point $p$ is a sum (difference)
of two points of rank $r_0$, so $r(p) \leq 2r_0$.
But every nonzero point has a scalar multiple in $B-B$ and rank is invariant under scalar multiplication.
\end{proof}

\section{Applications to Tensor Rank}\label{sec:tensors}
We now apply our bound on the maximum rank with respect to $X$ to various tensor ranks over $\C$ and $\R$.
We also discuss the relation between our bound and previously known bounds on the maximum tensor rank.
Note that there seems to be relatively little known about upper bounds for real tensor rank.
With rare exceptions previously known upper bounds on maximum rank are over $\C$,
while our Theorems also give the same upper bounds over $\R$.

\subsection{Symmetric Tensor Rank}

Symmetric tensors correspond to homogeneous polynomials (forms). The symmetric tensor rank of a homogeneous form $F$ of degree $d$ in $n$ variables (equivalently $n$-variate symmetric tensor of order $d$) 
is the least number $r$ of terms needed to write $F$ as a linear combination of $d$th powers of linear forms,
$$F = c_1 \ell_1^d + \dotsb + c_r \ell_r^d.$$
This corresponds precisely to the rank of a form $F$ with respect to the $d$-th Veronese variety $\nu_{d}(\PP^{n-1})$.
This is also known as the \defining{Waring rank} of $F$.
For example, since $xy = \frac{1}{4}(x+y)^2 - \frac{1}{4}(x-y)^2$,
the Waring rank of $xy$ is $2$
(as long as the characteristic of the field is not $2$).
See \cite{comonmour96,MR1735271,MR2865915} for introductions to Waring rank.
We limit our discussion to the fields $\R$ and $\C$.

We denote the maximum Waring rank $\rmax(n,d)$.
Classically, $\rmax(n,2) = n$ and $\rmax(2,d) = d$ are well-known,
but to our knowledge, only three other values of maximum rank over $\C$ are known:
the maximum rank of plane cubics is $5$, $\rmax(3,3) = 5$ \cite[\textsection96]{MR0008171}, \cite{comonmour96}, \cite[\textsection8]{Landsberg:2009yq};
the maximum rank of plane quartics is $7$, $\rmax(3,4) = 7$ \cite{Kleppe:1999fk,Paris:2013fk};
and the maximum rank of cubic surfaces is $7$, $\rmax(4,3) = 7$ \cite[\textsection97]{MR0008171}.

The vector space of forms of degree $d$ in $n$ variables has dimension $\binom{n+d-1}{n-1}$,
so trivially $\rmax(n,d) \leq \binom{n+d-1}{n-1}$ (by taking a basis consisting of powers of linear forms).
Several improvements are known: $\rmax(n,d) \leq \binom{n+d-1}{n-1} - n + 1$ \cite{Geramita,Landsberg:2009yq};
better, $\rmax(n,d) \leq \binom{n+d-2}{n-1}$ \cite{MR2383331}.
This was improved by Jelisiejew \cite{Jelisiejew:2013fk}:
\begin{equation}\label{eq: Jelisiejew bound}
  \rmax(n,d) \leq \binom{n+d-2}{n-1} - \binom{n+d-6}{n-3}
\end{equation}
and subsequently improved further by Ballico and De Paris \cite{Ballico:2013sf}:
\begin{equation}\label{eq: Ballico-DeParis bound}
  \rmax(n,d) \leq \binom{n+d-2}{n-1} - \binom{n+d-6}{n-3} - \binom{n+d-7}{n-3} .
\end{equation}

We denote the generic complex Waring rank---the Waring rank of a general form in $n$ variables of degree $d$, that is, one with general coefficients---by $\rgen(n,d)$.
Its value is given by the Alexander--Hirschowitz theorem
\cite{MR1311347}:
$\rgen(n,d) = \lceil \frac{1}{n} \binom{n+d-1}{n-1} \rceil$,
except if $(n,d) = (n,2)$, $(3,4)$, $(4,4)$, $(5,4)$, $(5,3)$.
In the first exceptional case, $\rgen(n,2) = n$.
For the rest, $\rgen(n,d) = \lceil \frac{1}{n} \binom{n+d-1}{n-1} \rceil + 1$,
and the secant variety $\sigma_{\rgen-1}(\nu_d(\PP^{n-1}))$ is a hypersurface.

We immediately obtain the following Corollary:

\begin{corollary}
The maximum real Waring rank of a real form of degree $d \geq 3$ in $n$ variables is at most:
\[
  \rmax(n,d) \leq 2\left\lceil \frac{1}{n} \binom{n+d-1}{n-1} \right\rceil,
\]
except $\rmax(3,4) \leq 11$, $\rmax(4,4) \leq 19$, $\rmax(5,4) \leq 29$, $\rmax(5,3) \leq 15$.
The same upper bound holds for the complex Waring rank.
\end{corollary}

Asymptotically, the upper bound \eqref{eq: Ballico-DeParis bound} is $nd/(n+d-1)$ times the generic rank.
Our bound is asymptotically better than \eqref{eq: Ballico-DeParis bound},
but worse for some small cases.
In the following table,
$\rmax^{J}$ denotes the upper bound \eqref{eq: Jelisiejew bound},
$\rmax^{BDP}$ denotes the upper bound \eqref{eq: Ballico-DeParis bound}
and $\rmax^*$ denotes our upper bound, $\rmax^* = 2 \rgen$.
The exact maximum is listed in the three cases where it is known.
\[
\begin{array}{ll r r r r r}
  \toprule
  n & d & \rgen & \rmax^J & \rmax^{BDP} & \rmax^* & \rmax \\
  \midrule
  3 & 3 &  4 &  5 &  5 &  8 & 5 \\
  3 & 4 &  6 &  9 &  8 & 11 & 7 \\
  3 & 5 &  7 & 14 & 13 & 14 & \\
  3 & 6 & 10 & 20 & 19 & 20 & \\
  3 & 7 & 12 & 27 & 26 & 24 & \\
  3 & 8 & 15 & 35 & 34 & 30 & \\
  \bottomrule
\end{array}
\qquad
\begin{array}{ll r r r r r}
  \toprule
  n & d & \rgen & \rmax^J & \rmax^{BDP} & \rmax^* & \rmax \\
  \midrule
  4 & 3 &  5 &   9 &   9 & 10 & 7 \\
  4 & 4 & 10 &  18 &  17 & 19 & \\
  4 & 5 & 14 &  32 &  30 & 28 & \\
  4 & 6 & 21 &  52 &  49 & 42 & \\
  4 & 7 & 30 &  79 &  75 & 60 & \\
  4 & 8 & 42 & 114 & 109 & 84 & \\
  \bottomrule
\end{array}
\]

\bigskip

\bigskip

Other than the exceptional cases of the Alexander--Hirschowitz theorem,
the hypersurface condition of Theorem~\ref{thm: refinement if secant hypersurface}
happens for Veronese varieties $X = \nu_d(\PP^{n-1})$ if and only if
\[
  (\rgen-1)n-1 = \binom{n+d-1}{n-1} - 2
\]
This happens if and only if $\binom{n+d-1}{n-1} \equiv 1 \pmod{n}$.
For example, if $n=2$ and $d$ is even then $\binom{n+d-1}{n-1} = d+1$ is odd;
other instances include $\binom{8}{2} \equiv 1 \pmod{3}$,
$\binom{11}{3} \equiv 1 \pmod{4}$, $\binom{14}{4} \equiv 1 \pmod{5}$.

\begin{example}
For binary $d$-forms the maximum rank (real or complex) is $\rmax = d$
and the generic rank is $\rgen = \lceil \frac{d+1}{2} \rceil$.
See \cite{MR2811260,Reznick:2013uq} for real binary forms with real (but not complex) Waring rank $d$.
In particular
$\rmax = 2\rgen - 2$ if $d$ is even, $\rmax = 2\rgen - 1$ if $d$ is odd,
so the upper bound of Theorem~\ref{thm:mainC} is almost sharp.

It is known that in the case $d$ is even, the $(\rgen-1)$st secant variety is a hypersurface,
defined by the vanishing of the determinant of the middle---that is, $(d/2)$th---catalecticant \cite{Sylvester:1851kx}.
Theorem~\ref{thm: refinement if secant hypersurface} then gives the upper bound $2\rgen - 1$,
which is still not sharp.

However it is not hard to show (and seems to be known to the experts)
that this hypersurface, which has degree $(d/2)+1$, has points of multiplicity $d/2$
precisely along the Veronese $\nu_d(\PP^1)$, the set of points of rank $1$.
Thus we recover, by adapting the proof of Theorem~\ref{thm: refinement if secant hypersurface},
the sharp bound $\rmax \leq 2\rgen - 2$.
\end{example}

\subsection{Tensor Rank}

Tensor rank of a tensor of order $d$ in $n$ variables corresponds precisely to the rank with respect to the Segre variety
$\operatorname{Seg}(\PP^n \times\dots \times \PP^n)$, with $d$ factors.
The generic tensor rank has been well-studied and it is known for several families of tensors.

\begin{example}[Tensors of format $2 \times \dotsb \times 2$]
The generic rank of tensors in $(\C^2)^{\otimes n} = \C^2 \otimes \dotsb \otimes \C^2$
is $\lceil \frac{2^n}{n+1} \rceil$ \cite{MR2762993}.
Therefore the maximum rank (real or complex) is at most $2 \lceil \frac{2^n}{n+1} \rceil$.
For $n > 8$ this is better than the bound $2^{n-2}$ given in \cite[Cor.~6.2]{Sumi:2013uq}
(see also \cite{Stavrou:2012kx} for the bound $3 \cdot 2^{n-3}$).
\end{example}

\begin{example}[Tensor rank in triple products]
It is known that the generic rank of tensors in $\C^n \otimes \C^n \otimes \C^n$
is $5$ when $n=3$, or $\lceil \frac{n^3}{3n-2} \rceil$ when $n>3$ \cite{MR798367}, \cite[Thm.~3.1.4.3]{MR2865915}.
Therefore by Theorem~\ref{thm:mainC} and Theorem~\ref{thm:mainR}
the maximum rank of tensors in $\C^n \otimes \C^n \otimes \C^n$
and the maximum real rank of tensors in $\R^n \otimes \R^n \otimes \R^n$
are at most $10$ when $n=3$, or $2\lceil \frac{n^3}{3n-2} \rceil \approx \frac{2}{3}n^2$ when $n>3$.

However this is not as good as previously known bounds.
It is known that the maximum rank of tensors in $\C^n \otimes \C^n \otimes \C^n$
and the maximum real rank of tensors in $\R^n \otimes \R^n \otimes \R^n$
are at most $\binom{n+1}{2}$ for $n \geq 4$, see \cite[Thm.~3.4, Thm.~3.7]{MR2652318} and \cite{MR545717,MR570374},
or $5$ when $n=3$ \cite{MR3089693}.

(See also \cite[Cor.~3.1.2.1]{MR2865915}, \cite[Cor.~3.5]{MR2996361}
and \cite[Cor.~6.7]{MR2854886}.)
\end{example}

\subsection{Waring problem with higher degree terms}
A variant of Waring rank is the number of terms needed to write a homogeneous polynomial of degree $kd$
as a linear combination of $k$th powers of $d$-forms.
See \cite{MR2935563,Reznick:2013yq,Carlini:2013vl}.
This is given by rank with respect to the projection of $\nu_k(\PP( \Sym^d \F^n))$ to $\PP( \Sym^{kd} \F^n)$
 given by the multiplication map $\Sym^k (\Sym^d \F^n) \to \Sym^{kd} \F^n$.
In \cite{MR2935563} it is shown that a generic complex $kd$-form in $n+1$ variables
is a sum of at most $k^n$ $k$th powers of $d$-forms,
and no fewer when $d$ is sufficiently large.
Therefore every complex $kd$-form is a sum of at most $2k^n$ $k$th powers of $d$-forms,
and every real $kd$-form is a real linear combination (possibly including negative coefficients)
of at most $2k^n$ $k$th powers of real $d$-forms.

\subsection{Antisymmetric Tensor Rank}

The rank of an alternating tensor $T \in \bigwedge^k \C^n$ is the least number of terms needed to write $T$ as a linear
combination of simple wedges. It is given by the rank with respect to a Grassmannian in its Pl\"ucker embedding.
In \cite{MR2846677} it is shown that the generic rank of an alternating tensor in $\bigwedge^3 \C^n$
is asymptotically $\frac{n^2}{18}$.
Therefore the maximum rank of such a tensor is asymptotically less than or equal to $\frac{n^2}{9}$.

\section*{Acknowledgements}

We thank E.~Ballico, J.~Buczy\'nski, and A.~de Paris for several very helpful comments.
We also thank the referee for several helpful comments, including suggesting the reference \cite{MR0008171},
and A.~Boralevi for providing us a copy of \cite{MR0008171}.
The first author was partially supported by Alfred P. Sloan Research Fellowship and NSF CAREER award DMS-1352073.

\bibliographystyle{amsalpha}
\renewcommand{\MR}[1]{{}}

\begin{thebibliography}{BCMT10}

\bibitem[{\AA}dl87]{MR947474}
Bj{\o}rn {\AA}dlandsvik, \emph{Joins and higher secant varieties}, Math. Scand.
  \textbf{61} (1987), no.~2, 213--222. \MR{947474 (89j:14030)}

\bibitem[AH95]{MR1311347}
J.~Alexander and A.~Hirschowitz, \emph{Polynomial interpolation in several
  variables}, J. Algebraic Geom. \textbf{4} (1995), no.~2, 201--222.
  \MR{1311347 (96f:14065)}

\bibitem[AL80]{MR570374}
M.~D. Atkinson and S.~Lloyd, \emph{Bounds on the ranks of some {$3$}-tensors},
  Linear Algebra Appl. \textbf{31} (1980), 19--31. \MR{570374 (81k:15028)}

\bibitem[AOP12]{MR2846677}
Hirotachi Abo, Giorgio Ottaviani, and Chris Peterson, \emph{Non-defectivity of
  {G}rassmannians of planes}, J. Algebraic Geom. \textbf{21} (2012), no.~1,
  1--20. \MR{2846677 (2012j:14068)}

\bibitem[AS79]{MR545717}
M.~D. Atkinson and N.~M. Stephens, \emph{On the maximal multiplicative
  complexity of a family of bilinear forms}, Linear Algebra Appl. \textbf{27}
  (1979), 1--8. \MR{545717 (80i:15019)}

\bibitem[Bal11]{MR2783180}
E.~Ballico, \emph{An upper bound for the {$X$}-ranks of points of {$\mathbb{P}^n$}
  in positive characteristic}, Albanian J. Math. \textbf{5} (2011), no.~1,
  3--10. \MR{2783180 (2012c:14108)}

\bibitem[Bal13]{Ballico:2013kx}
E.~Ballico, \emph{Symmetric tensor rank and scheme rank: An upper bound in
  terms of secant varieties}, Geometry \textbf{vol.~2013} (2013), 3 pages,
  Article ID 614195. doi:10.1155/2013/614195.

\bibitem[BP13]{Ballico:2013sf}
Edoardo Ballico and Alessandro~De Paris, \emph{Generic power sum decompositions
  and bounds for the {W}aring rank},
  \href{http://arxiv.org/abs/1312.3494}{\nolinkurl{arXiv:1312.3494}} [math.AG],
  Dec 2013.

\bibitem[BBCM11]{MR2895192}
Alessandra Bernardi, J{\'e}r{\^o}me Brachat, Pierre Comon, and Bernard
  Mourrain, \emph{Multihomogeneous polynomial decomposition using moment
  matrices}, I{SSAC} 2011---{P}roceedings of the 36th {I}nternational
  {S}ymposium on {S}ymbolic and {A}lgebraic {C}omputation, ACM, New York, 2011,
  pp.~35--42. \MR{2895192}

\bibitem[BBM12]{Bernardi:2012fk}
Alessandra Bernardi, J{{\'e}}r{{\^o}}me Brachat, and Bernard Mourrain, \emph{A
  comparison of different notions of ranks of symmetric tensors},
  \href{http://arxiv.org/abs/1210.8169v2}{\nolinkurl{arXiv:1210.8169}}
  [math.AG], Oct 2012.

\bibitem[BBS08]{MR2383331}
A.~Bia{\l}ynicki-Birula and A.~Schinzel, \emph{Representations of multivariate
  polynomials by sums of univariate polynomials in linear forms}, Colloq. Math.
  \textbf{112} (2008), no.~2, 201--233. \MR{2383331 (2009b:12006)}

\bibitem[BCG11]{MR2811260}
Mats Boij, Enrico Carlini, and Anthony~V. Geramita, \emph{Monomials as sums of
  powers: the real binary case}, Proc. Amer. Math. Soc. \textbf{139} (2011),
  no.~9, 3039--3043. \MR{2811260 (2012e:11070)}

\bibitem[BCMT10]{MR2736103}
Jerome Brachat, Pierre Comon, Bernard Mourrain, and Elias Tsigaridas,
  \emph{Symmetric tensor decomposition}, Linear Algebra Appl. \textbf{433}
  (2010), no.~11-12, 1851--1872. \MR{2736103 (2011k:15047)}

\bibitem[BCR98]{MR1659509}
Jacek Bochnak, Michel Coste, and Marie-Fran{\c{c}}oise Roy, \emph{Real
  algebraic geometry}, Ergebnisse der Mathematik und ihrer Grenzgebiete (3)
  [Results in Mathematics and Related Areas (3)], vol.~36, Springer-Verlag,
  Berlin, 1998, Translated from the 1987 French original, Revised by the
  authors. \MR{1659509 (2000a:14067)}

\bibitem[BH13]{MR3089693}
Murray~R. Bremner and Jiaxiong Hu, \emph{On {K}ruskal's theorem that every
  {$3\times3\times3$} array has rank at most 5}, Linear Algebra Appl.
  \textbf{439} (2013), no.~2, 401--421. \MR{3089693}


\bibitem[BL13]{MR2996361}
Jaros{\l}aw Buczy{\'n}ski and J.~M. Landsberg, \emph{Ranks of tensors and a
  generalization of secant varieties}, Linear Algebra Appl. \textbf{438}
  (2013), no.~2, 668--689. \MR{2996361}

\bibitem[Ble12]{Blekherman:2012dq}
Grigoriy Blekherman, \emph{Typical real ranks of binary forms},
  \href{http://arxiv.org/abs/1205.3257}{\nolinkurl{arXiv:1205.3257}} [math.AG],
  May 2012.

\bibitem[CGG11]{MR2762993}
Maria~Virginia Catalisano, Anthony~V. Geramita, and Alessandro Gimigliano,
  \emph{Secant varieties of {$\mathbb{P}^1\times\dots\times\mathbb{P}^1$} ({$n$}-times)
  are not defective for {$n\geq 5$}}, J. Algebraic Geom. \textbf{20} (2011),
  no.~2, 295--327. \MR{2762993 (2012b:14104)}

\bibitem[CM96]{comonmour96}
P.~COMON and B.~MOURRAIN, \emph{Decomposition of quantics in sums of powers of
  linear forms}, Signal Processing, Elsevier 53(2), 1996.

\bibitem[CO13]{Carlini:2013vl}
Enrico Carlini and Alessandro Oneto, \emph{Monomials as sums of {$k$}-th powers
  of forms}, \href{http://arxiv.org/abs/1305.4553}{\nolinkurl{arXiv:1305.4553}}
  [math.AC], May 2013.

\bibitem[DSS09]{MR2723140}
Mathias Drton, Bernd Sturmfels, and Seth Sullivant, \emph{Lectures on algebraic
  statistics}, Oberwolfach Seminars, vol.~39, Birkh\"auser Verlag, Basel, 2009.
  \MR{2723140 (2012d:62004)}

\bibitem[FOS12]{MR2935563}
Ralf Fr{\"o}berg, Giorgio Ottaviani, and Boris Shapiro, \emph{On the {W}aring
  problem for polynomial rings}, Proc. Natl. Acad. Sci. USA \textbf{109}
  (2012), no.~15, 5600--5602. \MR{2935563}

\bibitem[Fri12]{MR2854886}
Shmuel Friedland, \emph{On the generic and typical ranks of 3-tensors}, Linear
  Algebra Appl. \textbf{436} (2012), no.~3, 478--497. \MR{2854886
  (2012i:15040)}

\bibitem[Ger96]{Geramita}
Anthony~V. Geramita, \emph{Inverse systems of fat points: {W}aring's problem,
  secant varieties of {V}eronese varieties and parameter spaces for
  {G}orenstein ideals}, The {C}urves {S}eminar at {Q}ueen's, {V}ol.\ {X}
  ({K}ingston, {ON}, 1995), Queen's Papers in Pure and Appl. Math., vol. 102,
  Queen's Univ., Kingston, ON, 1996, pp.~2--114.

\bibitem[IK99]{MR1735271}
Anthony Iarrobino and Vassil Kanev, \emph{Power sums, {G}orenstein algebras,
  and determinantal loci}, Lecture Notes in Mathematics, vol. 1721,
  Springer-Verlag, Berlin, 1999, Appendix C by Iarrobino and Steven L. Kleiman.
  \MR{1735271 (2001d:14056)}

\bibitem[Jel13]{Jelisiejew:2013fk}
Joachim Jelisiejew, \emph{An upper bound for the {W}aring rank of a form},
  \href{http://arxiv.org/abs/1305.6957}{\nolinkurl{arXiv:1305.6957}} [math.AC],
  May 2013.

\bibitem[Jou83]{MR725671}
Jean-Pierre Jouanolou, \emph{Th\'eor\`emes de {B}ertini et applications},
  Progress in Mathematics, vol.~42, Birkh\"auser Boston Inc., Boston, MA, 1983.
  \MR{725671 (86b:13007)}

\bibitem[Kle99]{Kleppe:1999fk}
Johannes Kleppe, \emph{Representing a homogenous polynomial as a sum of powers
  of linear forms}, Master's thesis, University of Oslo, 1999,
  \url{http://folk.uio.no/johannkl/kleppe-master.pdf}.

\bibitem[Lan12]{MR2865915}
J.~M. Landsberg, \emph{Tensors: geometry and applications}, Graduate Studies in
  Mathematics, vol. 128, American Mathematical Society, Providence, RI, 2012.
  \MR{2865915}

\bibitem[Lic85]{MR798367}
Thomas Lickteig, \emph{Typical tensorial rank}, Linear Algebra Appl.
  \textbf{69} (1985), 95--120. \MR{798367 (87f:15017)}

\bibitem[LT10]{Landsberg:2009yq}
J.M. Landsberg and Zach Teitler, \emph{On the ranks and border ranks of
  symmetric tensors}, Found.\ Comp.\ Math. \textbf{10} (2010), no.~3, 339--366.

\bibitem[Par13]{Paris:2013fk}
Alessandro~De Paris, \emph{A proof that the maximal rank for plane quartics is
  seven},
  \href{http://arxiv.org/abs/1309.6475}{\nolinkurl{http://arxiv.org/abs/1309.6475}},
  Sep 2013.

\bibitem[Rez13a]{Reznick:2013yq}
Bruce Reznick, \emph{Forms as sums of powers of lower degree forms}, Slides
  from a talk at the SIAM Conference on Applied Algebraic Geometry, Fort
  Collins, Colorado, 2013.
  \url{http://www.math.uiuc.edu/~reznick/8213f-pm.pdf}, Aug 2013.

\bibitem[Rez13b]{Reznick:2013uq}
\bysame, \emph{On the length of binary forms}, Quadratic and Higher Degree
  Forms (New York) (K.~Alladi, M.~Bhargava, D.~Savitt, and P.~Tiep, eds.),
  Developments in Math., vol.~31, Springer, 2013, pp.~207--232.

\bibitem[Seg42]{MR0008171}
B.~Segre, \emph{The {N}on-singular {C}ubic {S}urfaces}, Oxford University
  Press, Oxford, 1942. \MR{0008171 (4,254b)}

\bibitem[SMS10]{MR2652318}
Toshio Sumi, Mitsuhiro Miyazaki, and Toshio Sakata, \emph{About the maximal
  rank of 3-tensors over the real and the complex number field}, Ann. Inst.
  Statist. Math. \textbf{62} (2010), no.~4, 807--822. \MR{2652318
  (2012e:15051)}

\bibitem[SS06]{MR2252121}
Bernd Sturmfels and Seth Sullivant, \emph{Combinatorial secant varieties}, Pure
  Appl. Math. Q. \textbf{2} (2006), no.~3, part 1, 867--891. \MR{2252121
  (2007h:14082)}

\bibitem[SSM13]{Sumi:2013uq}
Toshio Sumi, Toshio Sakata, and Mitsuhiro Miyazaki, \emph{Rank of tensors with
  size {$2 \times \dotsb \times 2$}},
  \href{http://arxiv.org/abs/1306.0708}{\nolinkurl{arXiv:1306.0708}} [math.RA],
  Jun 2013.

\bibitem[Sta12]{Stavrou:2012kx}
Stavros Stavrou, \emph{Canonical forms of {$2 \times 2 \times 2$} and {$2
  \times 2 \times 2 \times 2$} tensors}, Master's thesis, U.~Saskatchewan,
  2012.

\bibitem[Syl51]{Sylvester:1851kx}
J.J. Sylvester, \emph{An essay on canonical forms, supplement to a sketch of a
  memoir on elimination, transformation and canonical forms}, originally
  published by George Bell, Fleet Street, London, 1851. Paper 34 in {\it
  Mathematical Papers}, Vol.~1, Chelsea, New York, 1973, originally published
  by Cambridge University Press in 1904., 1851.


\end{thebibliography}
\providecommand{\bysame}{\leavevmode\hbox to3em{\hrulefill}\thinspace}
\providecommand{\MRhref}[2]{%
  \href{http://www.ams.org/mathscinet-getitem?mr=#1}{#2}
}
\providecommand{\href}[2]{#2}

\end{document}